\documentclass{article}
\usepackage{graphicx} 
\usepackage{mathtools}
\usepackage{amsmath}
\usepackage{amssymb}
\usepackage{amsfonts}
\usepackage{color}
\usepackage{lineno}
\usepackage{caption}
\usepackage {indentfirst}
\usepackage{graphicx}
\numberwithin{equation}{section}
\numberwithin{table}{subsection}
\usepackage{algorithm}
\usepackage{algorithmic,float}
\usepackage{multirow}
\usepackage{color}

\floatname{algorithm}{Algorithm}

\numberwithin{algorithm}{section}
\newtheorem{theorem}{Theorem}[section]
 
\newtheorem{Definition}[theorem]{Definition}

\newenvironment{proof}{{\noindent\it Proof}\quad}{\hfill $\square$\par}

\usepackage{algorithmic,float}

\makeatletter

\usepackage{geometry}
\usepackage[marginal]{footmisc}

\begin{document}

\title{Preprocessed GMRES for fast solution of linear equations\footnote{\noindent
  }}
 \author{Juan Zhang\thanks{Corresponding author: zhangjuan@xtu.edu.cn}, Yiyi Luo\thanks{Email: l19118069216@163.com}\\Hunan Key Laboratory for Computation and Simulation in Science and Engineering, \\Key Laboratory of Intelligent Computing and Information Processing of Ministry of Education,\\School of Mathematics and Computational Science, Xiangtan University}
\date{}

\maketitle
{\small{\bf Abstract.}The article mainly introduces preprocessing algorithms for solving linear equation systems. This algorithm uses three algorithms as inner iterations, namely RPCG algorithm, ADI algorithm, and Kaczmarz algorithm. Then, it uses BA-GMRES as an outer iteration to solve the linear equation system. These three algorithms can indirectly generate preprocessing matrices, which are used for solving equation systems. In addition, we provide corresponding convergence analysis and numerical examples. Through numerical examples, we demonstrate the effectiveness and feasibility of these preprocessing methods. Furthermore, in the Kaczmarz algorithm, we introduce both constant step size and adaptive step size, and extend the parameter range of the Kaczmarz algorithm to 
$\alpha\in(0,\infty)$. We also study the solution rate of linear equation systems using different step sizes. Numerical examples show that both constant step size and adaptive step size have higher solution efficiency than the solving algorithm without preprocessing.

{\small{\bf Keywords.} RPCG method; ADI method; Kaczmarz method; BA-GMRES method; adaptive step size.
\section{Introduction}
		Consider the   linear system to be solved \begin{equation}\label{eq0}Ax=b,\quad b\in\mathcal{R}(A),\end{equation}
		where $A\in \mathbb{R}^{m\times n}$ may be overdetermined or underdetermined, not necessarily full rank, $b\in \mathbb{R}^m$ may not be in $\mathcal{R}(A)$, and $\mathcal{R}(A)$ is the range of $A$. The above problem is equivalent to the following least squares problem \begin{equation}\label{eq1}\min_{x\in \mathbb{R}^n}\left\|b-Ax\right\|_2.\end{equation}
		
In the case of large sparse matrices, direct methods for solving this equivalent least squares problem \eqref{eq1} are typically expensive. Mature iterative methods include (preconditioned) CGLS \cite{refP3} and LSQR \cite{refP4} methods, which are mathematically equivalent to the conjugate gradient (CG) method applied to \eqref{eq0} \cite{refP3}, as well as the LSMR method \cite{refP2}. Generally, problem \eqref{eq0} has infinitely many solutions. If 
$x \in \mathcal{R}(A)$, then problem \eqref{eq1} is equivalent to 
$\min_{r_k\in R^n}\left\|Bb-BAx\right\|_2$, leading to iterations 
$x_k$ and residuals 
$r_k=b-Ax_k$, progressively reducing the residual at each step, where 
$B\in\mathbb{R}^{n\times m}$.

		It is worth noting that for ill-conditioned problems, the convergence speed of the iterative method may be slow because the condition number of 
$A^TA$ is the square of 
$A$, making preconditioning necessary. Therefore, for solving linear systems, internal iteration can be applied in Krylov subspace methods instead of explicitly applying a preconditioning matrix. This technique is commonly known as the inner-outer iteration method, which enhances the solution rate. As a result, scholars have developed preconditioned GMRES methods: BA-GMRES and AB-GMRES. Currently, BA-GMRES has been widely used in practical problems in various fields such as computer vision, computational fluid dynamics, computer graphics, etc. In BA-GMRES, the choice of preconditioning matrix significantly impacts the convergence speed and iteration count. Typically, the preconditioning matrix should approximate the original problem, and several methods can be used to generate it: directly using the original problem's solution, employing block matrix decomposition, utilizing Algebraic Multigrid (AMG) techniques, or applying preconditioning methods like Incomplete LU factorization (ILU). These approaches aim to reduce the condition number of the original problem and enhance algorithm convergence speed. Current research by numerous scholars has yielded significant theoretical advancements in understanding the performance of BA-GMRES.

In recent years, various scholars have proposed and studied different methods to solve the least squares problem using Krylov subspace techniques. Hayami, Yin, and Ito introduced the left preconditioned generalized minimum residual (BA-GMRES) method by applying GMRES to the least squares problem in \cite{refP10}, which employs a preconditioning matrix 
One potential drawback of these methods is the computational complexity and memory requirements, especially for large-scale problems. Additionally, the choice of preconditioner and tuning parameters can impact the convergence and performance of these methods, requiring careful optimization for different problem instances.
$B\in \mathbb{R}^{n\times m}$. In \cite{refP16}, Yin et al. explored Krylov subspace methods with preconditioning to solve large-scale sparse least squares problems and found that their preconditioned GMRES method yields better convergence than preconditioned CGLS and LSQR methods, especially for ill-conditioned problems. Wei and Ning presented a Jacobian-free approach for nonlinear least squares problems in \cite{refP15}, which is not dependent on the sparsity or structural pattern of the Jacobi, gradient, or Hessen matrices and can be applied to dense Jacobi matrices. Hansen et al. discovered in \cite{refP14} that AB-GMRES and BA-GMRES show similar semi-convergence behaviors as LSMR/LSQR, and they are suitable for large-scale CT reconstruction problems with noisy data and mismatched projector pairs. Zhao and Huang proved in \cite{refP17} that preconditioned iterative Krylov subspace methods with contraction and balance preconditioners outperform CGLS for ill-conditioned least squares problems. Finally, Du et al. proposed in \cite{refP25} a method that utilizes Kaczmarz as the inner iteration and F-AB-GMRES as the outer iteration to solve the preconditioned least squares problem, which demonstrates good performance. The research discussed various Krylov subspace methods for solving least squares problems, including preconditioned GMRES, Jacobian-free approaches, and different preconditioning strategies to enhance convergence for ill-conditioned problems. These methods have shown effectiveness in large-scale CT reconstruction and have demonstrated improved convergence performance compared to traditional methods like CGLS and LSQR. 
One potential drawback of these methods is the computational complexity and memory requirements, especially for large-scale problems. Additionally, the choice of preconditioner and tuning parameters can impact the convergence and performance of these methods, requiring careful optimization for different problem instances.

		This article primarily focuses on preprocessing the original matrix to generate a new Krylov subspace for solving linear equations. This approach significantly reduces the complexity of the algorithm. The process of generating the preprocessing matrix involves internal iteration, typically employing the Normalized Errors Gauss-Seidel (NE-GS) and Normalized Errors Successive Over-Relaxation (NE-SOR) methods\cite{refP11}, commonly referred to as Kaczmarz or row projection methods\cite{refP12, refP13}. Since their introduction in the 1930s, these methods have undergone extensive theoretical development and practical application. In 2006 and 2009, Strohmer and Vershynin proposed the Randomized Kaczmarz method, which exhibits expected exponential convergence rates, reigniting research interest in the Kaczmarz method\cite{refP18, refP19}. Bai, Z.-Z. and Wu introduced a more effective probabilistic criterion in \cite{refP21}, constructing the Greedy Randomized Kaczmarz method. Additionally, Popa summarized the convergence rates of Kaczmarz-type methods in \cite{refP22}, including greedy Kaczmarz\cite{refP12}, and random Kaczmarz methods\cite{refP24}.

Furthermore, two other iterative methods can be used for preprocessing: the Restricted Preconditioned Conjugate Gradient (RPCG) method and the Alternating Direction Implicit (ADI) iteration method, both of which can generate a matrix related to the preprocessing matrix. RPCG has been widely used for solving large-scale sparse linear equation systems. It is a variant of the Conjugate Gradient (CG) method that accelerates convergence by introducing a preconditioning matrix, approximating the solution to the original problem, which can be generated using various methods such as Incomplete LU (ILU) decomposition. RPCG is particularly effective in solving highly asymmetric or indefinite matrix linear systems. In 2003, Bai, Z.-Z. proposed the RPCG method for solving large sparse linear equation systems and proved its effectiveness compared to some standard preconditioned Krylov subspace methods such as the Preconditioned Conjugate Gradient and Preconditioned MINRES methods\cite{refP20}. Additionally, the ADI method has been widely applied in scientific computing, including linear systems, partial differential equations (PDEs), and optimization. Since the 1950s, Peaceman and Rachford proposed an alternating direction iteration algorithm based on the Peaceman-Rachford Splitting (PRS) method for solving second-order elliptic equations\cite{refP26}. Subsequently, Douglas and Rachford proposed an effective method, now known as the Douglas-Rachford Splitting (DRS) method, for solving heat conduction problems\cite{refP27}. Many other ADI methods have since been proposed and applied to solve various partial differential equations\cite{refP31}.

This article aims to combine the Restricted Preconditioned Conjugate Gradient method (RPCG), Alternating Direction Iteration method (ADI), Kaczmarz method, and its variants with the BA-GMRES method to develop an approach for solving linear equation systems, with the goal of improving the solution rate. The first part mainly introduces the algorithm flow of using the RPCG method as the inner iteration and BA-GMRES as the outer iteration, and its convergence analysis. The second part mainly introduces the algorithm flow of using the ADI method as the inner iteration and BA-GMRES as the outer iteration, and its convergence analysis. The third part mainly introduces the algorithm flow of using Kaczmarz and random Kaczmarz methods as the inner iterations, BA-GMRES as the outer iteration, and the corresponding convergence analysis. In addition, the complexity corresponding to the three methods is calculated. Finally, numerical examples and conclusions are provided for these algorithms.

	\section{RPCG-BA-GMRES Method}
		
		\subsection{Introduction to BA-GMRES}
		
		Firstly, let's introduce the GMRES method. According to the Hamilton-Cayley theorem, given $$f(A)=a_nA^n+\cdots+a_1 A+a_0 I=0, $$we have $A^{-1}=a_nA^{n-1}+\cdots+a_1I\coloneqq q_{n-1}(A)$. Directly calculating $x=A^{-1}b$ in the solution of $Ax=b$ is computationally expensive and may cause the inverse of the matrix to not exist.
		
		Therefore, we can construct the Krylov subspace: $\mathcal{K}_l(A,r)=(r,Ar,A^2r,\cdots,A^{l-1}r)$. Here, $l\ll n$ and when $l=n$, it becomes the previous $\mathbb{R}^n$ space.
		
		GMRES mainly solves $Ax=b$ by transforming it into a least squares problem \begin{equation}\min_{r\in \mathbb{R}^n}\left\|b-Ax\right\|.\end{equation}
		
		Define $r$ is the $n$-dimensional residual in the iterative process of $Ax=b$. We need to find a basis $\mathcal{K}_n(A,r)=span{r,Ar,\cdots,A^nr}$, which is orthogonalized through the Arnoldi process, and obtain $\mathcal{K}_l(A,r)\coloneqq span{q_1,q_2,\cdots,q_l}$ after the process. We then let $x_l\in \mathcal{K}_l(A,r)$, and we have $x_l=Q_ly$, where $y\in C_l,Q_l=(q_1,q_2,\cdots,q_l)$. At this point, the dimension of $y$ is much smaller than before. The least squares problem is equivalent to
		
		\begin{equation}\begin{aligned}\label{GMRESflow}\min_{x_l\in \mathcal{K}_l}\left\|r-Ax_l\right\|&\Rightarrow\min_{y\in C^l}\left\|r-AQ_ly\right\|\Rightarrow \min_{y\in C^l}\left\|r-Q_{l+1}\bar{H}_l y\right\|\\&\Rightarrow \min_{y\in C^l}\left\|Q^T_{l+1}r-\bar{H}_l y\right\|\Rightarrow \min_{y\in C^l}\left\|\beta e_1-\bar{H}_l y\right\|. \end{aligned}\end{equation}
		
		\eqref{GMRESflow} requires the Arnoldi process, that is, $Q_l^TAQ_l=H_l$.$Q_l$ is an orthogonal matrix, and $H$ is an upper Hessenberg matrix, then there is $AQ_l=Q_{l}H$. In GMRES, first we use the Arnoldi process to standardize orthogonalization of each group of bases to obtain an orthogonal matrix $Q_l$ to achieve matrix reduction. At this point, there is $x=x_0+Q_ly$. and the computation of $Q_ly$ is smaller than the computation of directly solving $Ax=b$ in $\mathcal{K}_n$.
		
		BA-GMRES and GMRES roughly deal with linear equations with the same idea, both reduce the coefficient matrix to obtain a krylov subspace with a lower dimension to solve the linear equation system. GMRES constructs the krylov subspace during the solution process only requires a matrix of residual and coefficients at each step to generate $$\mathcal{K}_l(A,r)=(r,Ar,\cdots,A^{l-1}r),$$ and then use this set of bases to generate standard orthogonal bases, and finally obtain a set of solutions: $x=x_0+V_ly$.
  
  In BA-GMRES, we use the preprocessing of the matrix $A^TAz=A^Tv_k$Thus we get a preprocessing matrix $B^{\ell_k}$($\ell_k$ is determined by the number of iterations of the inner iteration method). So our krylov subspace becomes $$\mathcal{Z}_m(A,z)=(z_1,z_2,\cdots,z_l).$$ Next, we introduce the generation of the preconditioning matrix $B$ in BA-GMRES. 
  \begin{theorem}
  If $ A \in \mathbb{R}^{n \times n}$  and $b \in \mathbb{R}^n$ , the preconditioning matrix $B$ of $Ax=b$ is $$B=\sum_{i=0}^{k}(M^{-1}N)^iM^{-1},$$ where $A = M-N$, and $k$ is any positive integer.\end{theorem}
  \begin{proof}
 
For $A^TAz = v_k$, if we split $A$ into two matrices $A = M-N$,
  \begin{equation}\begin{aligned}
		(M-N)z=v_k
		\Rightarrow M(I-M^{-1}N)z=v_k
		\Rightarrow z-M^{-1}Nz=M^{-1}v_k
	\end{aligned}.\end{equation} We can obtain the iterative formula for $z$ \begin{equation}\begin{aligned}\label{stable inner iteration}z_{k}&=M^{-1}Nz_{k-1}+M^{-1}v_k\\&=M^{-1}N(M^{-1}Nz_{k-2}+M^{-1}v_k)+M^{-1}v_k\\
&=(M^{-1}N)^2z_{k-2}+M^{-1}NM^{-1}v_k+M^{-1}v_k\\
&=(M^{-1}N)^kz_0+\sum_{i=0}^{k-1}(M^{-1}N)^iM^{-1}v_k\\
&=\sum_{i=0}^{k}(M^{-1}N)^iM^{-1}v_k.
\end{aligned}\end{equation} where we take $z_0 = M^{-1}v_k$.     
  \end{proof}
	
	We have $B^{\ell_k}=\sum_{i=0}^{k}(M^{-1}N)^iM^{-1}$. In the left preconditioning process, we can transform it into solving $B^{\ell_k}A$ instead of directly solving the preconditioning matrix $B^{\ell_k}$. Let $H = M^{-1}N$, then $A = M(I-H)$, and we have \begin{equation}B^{\ell_k}A = I-H^k.\end{equation}
	
	In the right-hand vector $v_k$ corresponding to $A^TAz = A^Tv_k$, we can use the Kaczmarz iteration to obtain $z$. And in the iterative process of solving $z$ vector, we can obtain the maximum iteration number $\ell_{max}$. Then, we can generate the smallest $\ell$ such that $\left\|B^\ell Av_k-v_k\right\|\leqslant\eta\left\|v_k\right\|$, and set $\ell_k=\min\{\ell_{max},\ell\}$. At this point, the preconditioning matrix $B^{\ell_k}$ can be determined, and we can obtain the new Krylov subspace vector $z_k = B^{\ell_k}v_k$. The previous Krylov subspace is defined as $V$, and the subsequent Krylov subspace is defined as $Z$. As shown above, a new vector $z_k$ can be obtained at each iteration. Then, we use the Arnoldi process to orthogonally normalize this space, and obtain $Z_l=(z_1,z_2,\cdots,z_l)$, which is a orthogonal matrix with many more rows than columns. Finally, we can obtain the true solution using $x=x_0+Z_my$. The algorithm is as follows.
		
		\begin{algorithm}[H]
			\caption{BA-GMRES method\cite{refP1}.}
			\begin{algorithmic}
				\STATE 1. Set initial values $x_0$ and $r_0=b-Ax_0$.
				
				\STATE 2. Apply $l$ iterations of stable iteration $A^TAz=A^Tr_0$ to obtain $z_0=B^{(l)}r_0$.
				
				\STATE 3. $\beta=\left\|z_0\right\|_2,v_1=z_0/\beta$
				
				\STATE 4. For $k=0,1,2\cdots,$ until convergence, do
				
				\STATE 5. $u_k=Av_k$.
				
				\STATE 6. Apply $l$ iterations of stable iteration $A^TAz=A^Tu_k$ to obtain $z_k=B^{(l)}u_k$.
				
				\STATE 7. For $i=1,2,\cdots,k$, do
				
				\STATE 8. $h_{i,k}=(z_k,v_i),z_k=z_k-h_{i,k}v_i$
				
				\STATE 9. End for
				
				\STATE 10. $h_{k+1,k}=\left\|z_k\right\|_2,v_{k+1}=z_k/h_{k+1,k}$
				
				\STATE 11. End for
				
				\STATE	12.$y_k=argmin_{y\in R^k}\left\|\beta e_1-\bar{H}_ky\right\|_2,x_k=x_0+[v_1,v_2,\cdots,v_k]y_k$
				
			\end{algorithmic}
		\end{algorithm}

	\subsection{Convergence Analysis of BA-GMRES}
		\begin{Definition}\rm(\cite{refP1}, definition 4.1)
			A matrix $H$ is semi-convergent if and only if $\lim_{i\to\infty}H^i$ exists.
		\end{Definition}
		
		\begin{theorem}\label{th1}
			\rm(\cite{refP1}, Theorem 4.6) Suppose $H$ is semi-convergent. By using the formula \eqref{stable inner iteration}, the preconditioning matrix $B^{\ell_k}$ is determined. The minimum least squares solution of $\left\|b-Ax\right\|_2$ is finally determined for all $b\in \mathbb{R}^m$ and $x_0\in \mathbb{R}^n$.
		\end{theorem}
		
		For the convergence analysis of the BA-GMRES method, we provide a convergence bound for BA-GMRES.
		
		\begin{theorem}\label{th2}
		\rm(\cite{refP1}, Theorem 4.9) Let $r_k$ be the k-th residual of BA-GMRES preprocessed by the $B^{\ell_k}$ inner iteration, and let $T$ be the Jordan basis of $B^{\ell_k}A$. Suppose $H$ is semi-continuous. Then, we have
			
			\begin{equation}
				\left\|B^{\ell_k}r_k\right\|_2\leqslant\kappa(T)\sum_{i=0}^{\tau(\ell_k,d)} \genfrac ( ) {0pt}{0}{k}{i}\rho(H)^{k\ell_k-i}\left\|B^{(\ell_k)}r_0\right\|_2.
			\end{equation}
			
			for all $x_0\in \mathbb{R}(B^{\ell_k})$ and $b\in \mathbb{R}^m$, where $\kappa(T)=\left\|T\right\|_2\left\|T^{-1}\right\|_2$, $d$ is the order of the largest Jordan block corresponding to $B^{\ell_k}A$, and $\tau(\ell_k,d)=\min(\ell_k,d-1)$.
		\end{theorem}
		
		\begin{proof}
			Theorem(\ref{th1}) ensures that the minimum least squares solution of $\min_{x\in \mathbb{R}^n}\left\|b-Ax\right\|_2$ is determined by the BA-GMRES preprocessed by inner iteration for all $x_0\in \mathbb{R}(B^{\ell_k})$ and $b\in \mathbb{R}^m$. In \cite{refP2}, we have
			
			\begin{equation}
				\left\|B^{\ell_k}r_k\right\|_2=\mathop{\min}_{
					p\in P_k
					\atop
					p(0)=1}\left\|p(B^{\ell_k}A)B^{\ell_k}r_0\right\|_2\leqslant\kappa(T)\left(\mathop{\min}_{p\in P_k
					\atop
					p(0)=1}\max_{1\leqslant i\leqslant s}\left\|p(J_i)\right\|_2\right)\left\|B^{\ell_k}r_0\right\|_2.
			\end{equation}
			
			where $P_k$ is the set of all polynomials of degree at most $k$, $J_i$ is the Jordan block corresponding to a non-zero eigenvalue of $B^{(\ell_k)}A,i=1,2,\cdots,s$. The bounds for $\mathop{\min}{p\in P_k
				\atop
				p(0)=1}\max{1\leqslant i\leqslant s}\left\|p(J_i)\right\|_2$ are given in Theorem 4.8, Theorem 2, and Theorem 5 in \cite{refP2}.
			
			\begin{equation}\mathop{min}_{p\in P_k
									\atop
									p(0)=1}\max_{1\leqslant i\leqslant s}\left\|p(J_i)\right\|_2\leqslant\sum_{i=0}^{\tau(k, d)}(\genfrac ( ) {0pt}{0}{k}{i}\rho(H)^{kl-i}. \end{equation}
			
			Therefore, the theorem is proved.
		\end{proof}

	We use the RPCG method as the inner iteration and the BA-GMRES method as the outer iteration to solve the equation $Ax=b$. In the following, we introduce the RPCG method.
	
	\subsection{Introduction to RPCG Method}
	
	Let $A=PHQ$, where $H\in \mathbb{R}^{n\times n}$ is a symmetric positive definite matrix, and $P$ and $Q$ are non-singular matrices. Let $W=Q^{-1}P^T$. Assume that $M=PGQ$ is the preconditioner of the matrix $A\in \mathbb{R}^{n\times n}$, where $G\in \mathbb{R}^{n\times n}$ is a symmetric positive definite matrix. Since $G$ is a symmetric positive definite matrix, there exists a non-singular matrix $S\in \mathbb{R}^{n\times n}$ such that $G=S^TS$. Furthermore, if we let
	\begin{equation}
		\mathbf{x}=(SQ)x\quad \text{and} \quad \mathbf{b}=(PS^T)^{-1}b,
	\end{equation}
	then the linear system \eqref{eq0} can be equivalently written as
	\begin{equation}\label{eq2}
		\mathbf{Rx}=\mathbf{b}, \quad \text{where} \quad \mathbf{R}=(PS^T)^{-1}A(SQ)^{-1}.
	\end{equation}
	Clearly, $\textbf{R}=(PS^T)^{-1}A(SQ)^{-1}=S^{-T}HS^{-1}$ is a symmetric positive definite matrix, so we can use the Conjugate Gradient (CG) method to iteratively solve the linear system \eqref{eq2} and obtain the RPCG method, which fully utilizes the structure of the original matrix. First, we give the specific algorithmic process of the Conjugate Gradient method.
	
	\begin{algorithm}[H]
		\begin{algorithmic}
			\caption{CG Method for Solving $Rx=b$}
			\label{al1}
			\STATE Step 1:Choose initial value $x_0\in \mathbb{R}^n$
			\STATE Step 2:for $k=0,1,2,\cdots$
			\STATE Step 3:$\alpha_k=-\frac{\boldsymbol{r}k^T\boldsymbol{r}k}{\boldsymbol{p}k^T\mathbf{R}\boldsymbol{p}k}$
			\STATE Step 4:$\mathbf{x}_{k+1}=\mathbf{x}_k-\alpha_k\boldsymbol{p}k$
			\STATE Step 5:$r{k+1}=r_k+\alpha_kRp_k$
			\STATE Step 6:$\beta_k=\frac{r{k+1}^Tr{k+1}}{r{k}^Tr_k}$
			\STATE Step 7:$p{k+1}=r_{k+1}+\beta_kp_k.$
		\end{algorithmic}
	\end{algorithm}
	
	The above method \ref{al1} can be redefined using information about the original linear system \eqref{eq1}. Given an initial vector $x_0\in \mathbb{R}^n$, let $r_0=b-Ax_0$, $z_0=M^{-1}r_0$, $p_0=z_0$, $v_0=W^{-1}z_0$, and $q_0=v_0$. Define
	\begin{equation}\label{eq3}
			\begin{cases}
				x_k=(SQ)^{-1}\boldsymbol{x}_k,\quad r_k=(PS^T)\boldsymbol{r}_,\quad z_k=(SQ)^{-1}\boldsymbol{r}_k,\\
				p_k=(SQ)^{-1}\boldsymbol{p}_k,\quad v_k=W^{-1}z_k,\quad q_k=W^{-1}p_k.
			\end{cases}
	\end{equation}
	Then we have
	\begin{equation}\label{eq4}\boldsymbol{r_0} = \boldsymbol{b-Rx_0}=(PS^T)^{-1}(b-Ax_0)=(PS^T)^{-1}r_0,\end{equation}and
		\begin{equation}\label{eq40}r_k=(PS^T)\boldsymbol{r_k}=(PS^T)r_k=(PS^T)(SQ)z_k=Mz_k,
	\end{equation}Moreover Since
	\begin{equation}\label{eq41}\boldsymbol{r}_k^T\boldsymbol{r}_k=(SQz_k)^T(SQz_k)=z_k^TQ^TP^{-1}Mz_k=v_k^Tr_k,\end{equation}
	and
	\begin{equation}\label{eq44}\boldsymbol{p}_k^T\boldsymbol{R}\boldsymbol{p}_k=(SQp_k)^T(PS^T)^{-1}A(SQ)^{-1}(SQp_k),\end{equation}
	We can get \begin{equation}\label{eq42}\alpha_k=-\frac{\boldsymbol{r}_k^T\boldsymbol{r}_k}{\boldsymbol{p}_k^T\boldsymbol{R}\boldsymbol{p}_k}=-\frac{v_k^Tr_k}{q_k^TAp_k},\end{equation}
	and
	\begin{equation}\label{eq43}\beta_k=\frac{\boldsymbol{r}_{k+1}^T\boldsymbol{r}_{k+1}}{\boldsymbol{r}_k^T\boldsymbol{r}_k}=\frac{v_{k+1}^Tr_{k+1}}{v_k^Tr_k}.\end{equation}
	we can rewrite Algorithm \ref{al1} in terms of \eqref{eq42} and \eqref{eq43} to obtain the RPCG algorithm.
	
	\begin{algorithm}[H]
		\begin{algorithmic}
			\caption{RPCG Method for S solving $Ax=b$}
			\label{al2}
			\STATE Step 1: Choose initial value $x_0\in \mathbb{R}^n$
			\STATE Step 2:$r_0=b-Ax_0, \quad z_0=M^{-1}r_0, \quad p_0=z_0, \quad v_0=W^{-1}z_0, \quad q_0=v_0$
			\STATE Step 3:for $k=0,1,2,\cdots$
			\STATE Step 4:$\alpha_k=-\dfrac{v_k^Tr_k}{q_k^TGq_k}$
			\STATE Step 5:$x_{k+1}=x_k+\alpha_kp_k$
			\STATE Step 6:$r_{k+1}=r_k+\alpha_kAv_k$
			\STATE Step 7:$z_{k+1}=M^{-1}r_{k+1}$
			\STATE Step 8:$v_{k+1}=W^{-1}z_{k+1}$
			\STATE Step 9:$\beta_k=\dfrac{v_{k+1}^Tr_{k+1}}{v_k^Tr_k}$
			\STATE Step 10: $p_{k+1}=z_{k+1}+\beta_kp_k$
			\STATE Step 11:$q_{k+1}=v_{k+1}+\beta_kq_k.$
		\end{algorithmic}
	\end{algorithm}

where $W=Q^{-1}P^T, P, Q\in  \mathbb{R}^{n\times n}$ are two non-singular matrices,  where $H\in  \mathbb{R}^{n\times n}$, $A= PHQ$, $H$ is symmetrical positive,  $M = PGQ$ is $A$ pre-regulator,  $G$ is $H$ approximation. Note that Algorithm \ref{al2} is equivalent to the CG method applied to the transformed system \eqref{eq2}. The RPCG method has the same convergence properties as the CG method, but it is more efficient for solving large linear systems with a specific structure, such as those arising from finite element methods.

From the above analysis, it can be seen that finding the two non-singular matrices of $P$ and $Q$ is the key to the RPCG method. Therefore, we use the elementary transformation of the chunked matrix to obtain the specific expressions of the $P$ and $Q$ matrices.
First, let's divide the coefficient matrix into $$A=\begin{bmatrix}
	T&B\\C&D
\end{bmatrix}.$$
Next, 
$$
\overbrace{\begin{bmatrix}
		I&0\\-CT^{-1}&I
		\end{bmatrix}}^{P^{-1}}
	\overbrace{\begin{bmatrix}
	T&B\\C&D
\end{bmatrix}}^{A}\overbrace{\begin{bmatrix}
		I&-T^{-1}B\\0&I
		\end{bmatrix}}^{Q^{-1}}
=\overbrace{\begin{bmatrix}
	T&0\\0&S
	\end{bmatrix}}^{H}. $$ $$P=\begin{bmatrix}
I&0\\CT^{-1}&I
\end{bmatrix}, Q=\begin{bmatrix}
I&T^{-1}B\\0&I
\end{bmatrix},$$
where $S=D-CT^{-1}B$ is $A$'s Schur complement matrix.

Finally, let's find the symmetric positive definite approximation matrix $G$ of the matrix $H$. Since the $T$ chunks in the matrix are symmetrically positively determined, there will be $T=E^TE$, where the $E$ matrix is an upper triangular matrix.
Extract the elements of the upper triangular portion of $S$ and add them to the lower trigonometric to make it a symmetric matrix $\hat{S}$. Then there is $\hat{S}=F^TF$,  $$G=\begin{bmatrix}
	E&0\\0&F
\end{bmatrix}^T\begin{bmatrix}
E&0\\0&F
\end{bmatrix}.$$

 Combined with the BA-GMRES method, the specific algorithm flow obtained is as follows.

 \begin{algorithm}[H]
	\caption{RPCG-BA-GMRES algorithm flow}\label{tab:算法流程}
	\begin{algorithmic}
		\STATE	Take the initial value $x_0$,$r=b-Ax_0$
		\STATE	$v_1=r/\left\|r\right\|_2$	
		\STATE	for $k=1,2\cdots$until convergence do
		\STATE   Apply RPCG methods to solve $Az=v_k$
		\STATE	Compute $$B^{l_k}=(M^{-1}N)^kz_0+\sum_{i=0}^{k-1}(M^{-1}N)^iM^{-1}$$
		where $A=M-N$, $l_{max}$ is the maximum number of iterations obtained by solving $Az=v_k$ by the RPCG method,$l_k=\min\{l_{max},l\}$, $l$ is satisfied
  $\left\|BAv_k-v_k\right\|_2\leqslant\left\|v_k\right\|_2$
		\STATE	Solve $z_k=B^{l_k}v_k$
		\STATE	 $w_k=BAz_k$
		\STATE	for i=1,2$\cdots,k$,do
		\STATE	$h_{i,k}=w_k^Tv_i,w_k=w_k-h_{i,k}v_i$
		\STATE	end for
		\STATE	11:$h_{k+1,k}=\left\|w_k\right\|_2,v_{k+1}=w_k/h_{k+1,k}$
		\STATE	end for
		\STATE	$y_k=argmin_{y\in R^k}\left\|\beta e_1-\bar{H}_ky\right\|_2,$
		\quad $u_k=[v_1,v_2,\cdots,v_k]y_k,$
		\quad where $\bar{H}_k=\{h_{ij}\}_{1\leqslant i\leqslant k+1;1\leqslant j\leqslant k}$
		\STATE	$x_k=x_0+u_k$
	\end{algorithmic}
\end{algorithm}

	\subsection{Convergence Analysis of RPCG Method}
	
	\begin{theorem}
		Let $A\in\mathbb{R}^{n\times n}$ be a non-singular matrix, and let there exist two non-singular matrices $P$ and $Q$ such that $A=PHQ$ where $H\in \mathbb{R}^{n\times n}$ is a symmetric positive definite matrix. Let $M=PGQ$, where $G$ is a symmetric positive definite matrix. If the RPCG method starts from an initial vector $x_0\in \mathbb{R}^n$ and generates an approximate solution $x_k$ to the linear system \eqref{eq1} after $k$ iterations, then the error $\left\|x_k-x^*\right\|$ satisfies the following inequality.
		\begin{equation}
			\left\|x_k-x^*\right\|_{W^{-T}A}\leqslant2\left(\frac{\sqrt{\kappa(M^{-1}A)}-1}{\sqrt{\kappa(M^{-1}A)}+1}\right)^k\left\|x_0-x^*\right\|_{W^{-T}A}.
		\end{equation}
		for any $x\in \mathbb{R}^n$ and any symmetric positive definite matrix $X\in \mathbb{R}^{n\times n}$, $\left\|x\right\|_X=\sqrt{(x,Xx)}=\sqrt{x^TXx}$, and $\kappa(M^{-1}A)$ represents the Euclidean condition number of the matrix $M^{-1}A$.
	\end{theorem}
	
	\begin{proof}
		Since the convergence rate of the conjugate gradient method satisfies the inequality
		\begin{equation}\label{eq01}
			\left\|x_k-x^*\right\|_A\leqslant2\left(\frac{\sqrt{\kappa(A)}-1}{\sqrt{\kappa(A)}+1}\right)^k\left\|x_0-x^*\right\|_A,
		\end{equation}
		we only need to prove that $\kappa(W^{-T}A)=\kappa(M^{-1}A)$.
		
		Since $W^{-T}A=(Q^TP^{-1})(PHQ)=Q^THQ$, where $Q\in \mathbb{R}^{n\times n}$ is a non-singular matrix and $H\in \mathbb{R}^{n\times n}$ is a symmetric positive definite matrix, we know that $W^{-T}A$ is a symmetric positive definite matrix. Moreover, since $G=L^TL$ and $L$ is a non-singular matrix, we have
		$$\begin{aligned}M^{-1}A=&(PGQ)^{-1}(PHQ)=Q^{-1}G^{-1}HQ=(LQ)^{-1}(L^{-T}HL^{-1})(LQ)=(LQ)^{-1}R(LQ)\end{aligned}$$
		Therefore, $\kappa(M^{-1}A)=\kappa(R)$. Additionally, for any $x\in \mathbb{R}^n$ satisfying $x=(SQ)x$, we have
		$$\begin{aligned}\left\|\textbf{x}\right\|_A^2=\left\|\textbf{x}\right\|_\textbf{R}^2&=(\textbf{x},\textbf{Rx})=((LQ)x,((PL^T)^{-1}A(LQ)^{-1})(LQ)x)\\&=((LQ)x,(PS^T)^{-1}Ax)=x^TW^{-T}Ax=\left\|\right\|^2_{W^{-T}A}.\end{aligned}$$
		Now the conclusion of this theorem directly follows by applying the standard convergence theorem to the linear system \eqref{eq01}.
	\end{proof}
	
	\section{ADI-BA-GMRES Method}
 \subsection{Introduction to ADI Method}
	
	First, perform left preconditioning on the linear system $Ax=b$ and get $FAx=Fb$, where $F$ is the diagonal of the coefficient matrix $A$, and if any diagonal element is zero, it is replaced with $1$ to ensure that $F$ is invertible. This simple diagonal preconditioning can greatly improve convergence speed \cite{refP28}. The linear system becomes $\hat{A}x=\hat{b}$.
	
	Considering two forms of the split of $\hat{A}$:
	\begin{equation}
	\begin{aligned}
		\hat{A}&=H+\alpha I-(\alpha I-S),\\
		\hat{A}&=S+\alpha I-(\alpha I-H),
	\end{aligned}
	\end{equation}
	where $I$ represents the identity matrix.
	
	Given an initial guess $x^0$, compute a sequence ${x^k}$ using the iteration formula:
	\begin{equation}\label{eq16}
		\begin{cases}
			(H+\alpha I)x^{k+\frac{1}{2}}=(\alpha I-S)x^k+\hat{b}\\
			(S+\alpha I)x^{k+1}=(\alpha I-H)x^{k+\frac{1}{2}}+\hat{b}.
		\end{cases}
	\end{equation}
	where $\alpha\in(0,2)$, $H=\frac{1}{2}(A+A^T)$, $S=\frac{1}{2}(A-A^T)$.
	
	Combined with the BA-GMRES algorithm, the specific algorithm is as follows.
	
	\begin{algorithm}[H]
		\caption{ADI-BA-GMRES Algorithm}\label{tab:algorithm}
		\begin{algorithmic}[1]
			\STATE Take the initial value $x_0$ and compute the residual $r=b-Ax_0$
			\STATE Set $v_1=r/\left\|r\right\|_2$
			\STATE For $k=1,2,\cdots$ until convergence, do
			\STATE Use the ADI method to solve $Az=v_k$
			\STATE Compute $B^{\ell_k}=(M^{-1}N)^kz_0+\sum_{i=0}^{k-1}(M^{-1}N)^iM^{-1}$, where $A=M-N$, $\ell_{max}$ is the maximum number of iterations obtained by the ADI method to solve $Az=v_k$, $\ell_k=\min\{\ell_{max},\ell\}$, and $\ell$ is the smallest value found to satisfy $\left\|BAv_k-v_k\right\|_2\leqslant\left\|v_k\right\|_2$
			\STATE Solve $z_k=B^{\ell_k}v_k$
			\STATE Compute $w_k=BAz_k$
			\STATE For $i=1,2\cdots,k$, do
			\STATE Set $h{i,k}=w_k^Tv_i$ and $w_k=w_k-h{i,k}v_i$
			\STATE End for
			\STATE Set $h_{k+1,k}=\left\|w_k\right\|_2$ and $v{k+1}=w_k/h_{k+1,k}$
			\STATE End for
			\STATE Set $y_k=\arg\min_{y\in R^k}\left\|\beta e_1-\bar{H}_ky\right\|_2$, $u_k=[v_1,v_2,\cdots,v_k]y_k$, where $\bar{H}_k=\{h_{ij}\}_{1\leqslant i\leqslant k+1;1\leqslant j\leqslant k}$
			\STATE Set $x_k=x_0+u_k$
		\end{algorithmic}
	\end{algorithm}
 
	\subsection{Convergence Analysis of the ADI Method}
	To analyze the convergence of \eqref{eq16}, we eliminate the intermediate vector $x^{k+\frac{1}{2}}$ and rewrite the iteration in the form of a fixed-point equation:
	\begin{equation}\label{eq17}
		x^{k+1}=\mathcal{T}_{\alpha}x^k+c,
	\end{equation}
where $\mathcal{T}_{\alpha}\coloneqq(S+\alpha I)^{-1}(\alpha I-H)(H+\alpha I)^{-1}(\alpha I-S)$ is the iteration matrix of the method and 

$$c\coloneqq(S+\alpha I)^{-1}\left[ (\alpha I-H)(H+\alpha I)^{-1}+I\right]b.$$ 
The fixed-point iteration equation\eqref{eq17} converges for any initial guess $x_0$ and right-hand side $b$ to the solution 
$x=A^{-1}b$ if and only if $\rho(\mathcal{T}_{\alpha})\leq 1$, 
where $\rho(\mathcal{T}_{\alpha})$ denotes the spectral radius of $\mathcal{T}_{\alpha}$.
	
	\begin{theorem}
		Consider the problem \eqref{eq1}.Since $H$ is symmetric,$S$ is skew-symmetric,and $F$ is a non-singular diagonal matrix, the iteration scheme \eqref{eq16} converges unconditionally, i.e., $\rho(\mathcal{T}_{\alpha})<1$ for all $\alpha>0$.
	\end{theorem}
	
	\begin{proof}
		Note that for any $\alpha>0$, $S+\alpha I$ and $H+ \alpha I$ are non-singular matrices. Thus, $\mathcal{T}_{\alpha}$ is similar to
		$$\hat{\mathcal{T}}_{\alpha}\coloneqq(\alpha I-H)(\alpha I+H )^{-1}(\alpha I-S)(\alpha I+S)^{-1}.$$
		Then, we have
		\begin{equation}\begin{aligned}
			\rho(\mathcal{T}_{\alpha})=&\rho((S + \alpha I)^{-1}(\alpha I-H)(H + \alpha I)^{-1}(\alpha I-S))\\
			\leqslant&\rho((\alpha I-H)(\alpha I+H )^{-1}(\alpha I-S)(\alpha I+S)^{-1})\\
			\leqslant&\left\|(\alpha I-H)(\alpha I+H )^{-1}\right\|_2\left\|(\alpha I-S)(\alpha I+S)^{-1}\right\|_2,
		\end{aligned}\end{equation}
		where $\left\|\cdot\right\|_2$ denotes the spectral norm.
		
		Since $S$ is skew-symmetric, we have $S^T=-S$. Let $S_{\alpha}=(\alpha I-S)(\alpha I+S)^{-1})$, then we have
		\begin{equation}\begin{aligned}
			S_{\alpha}^TS_{\alpha}&=(\alpha I-S)^{-1}(\alpha I+S)(\alpha I-S)(\alpha I+S)^{-1}\\
			&=(\alpha I-S)^{-1}(\alpha I-S)(\alpha I+S)(\alpha I+S)^{-1}\\&=I.
		\end{aligned}\end{equation}
		Hence, we can see that $\left\|(\alpha I-S)(\alpha I+S)^{-1}\right\|_2=1$. Then, by using $\lambda_i\in\lambda(H)>0,i=1,2,\cdots,n$ and $\alpha>0$, we have
		\begin{equation}\rho(\mathcal{T}_{\alpha})\leqslant\left\|(\alpha I-H)(\alpha I+H )^{-1}\right\|_2=\max_{\lambda_i\in\lambda(H)}\left\|\frac{\alpha-\lambda_i}{\alpha+\lambda_i}\right\|<1.\end{equation}
		Therefore, the theorem is proved.
\end{proof}

\section{Kaczmarz-BA-GMRES Method}
\subsection{Introduction to Kaczmarz Method}
The Kaczmarz algorithm is an iterative method for solving linear systems $Ax=b$. It has a linear convergence rate and requires only $\mathcal{O}(n)$ operations and storage per iteration. In each iteration $k$, the algorithm selects a row $r_{i_k}\in[m]$ from the coefficient matrix and performs an orthogonal projection of the current estimate $x^k$ onto the hyperplane defined by the equation $a_{i_k}^Tx =b_{i_k}$:
\begin{equation}\min_{x}\left\|x-x^k\right\|^2,\quad s.t\quad a_{i_k}^Tx=b_{i_k}.\end{equation}
Then, the algorithm chooses the next iteration point along the line connecting the current estimate and the projection. Geometrically, this corresponds to selecting a plane defined by one row of the coefficient matrix, finding the orthogonal projection of the current estimate onto this plane, and then finding the intersection of this line with another plane defined by a different row of the coefficient matrix. As more equations are added, the feasible region shrinks.

The basic Kaczmarz iteration formula is given by
\begin{equation}x^{k+1}=x_k+\alpha\frac{a_i^Tx_k-b_i}{a_i^Ta_i}a_i,\end{equation}
where $a_i$ is the $i$-th column of $A$. The choice of rows is a key factor in the algorithm. In this paper, we consider cyclic and random selection of rows. The results obtained depend on which set of rows $J\subseteq[m]$ is selected. Some rows may be more important than others, for example, in the case where $T\subseteq[m]$ such that $\mathcal{X}={x\in R^n:A_Tx=b_T}$, where $A_T$ denotes the submatrix of $A$ with rows indexed by $T$. In this case, the rows in $T$ are more important than those not in $T$. If there are multiple sets $T$, then the rows of the matrix $A_T$ with the best conditions are more important than other rows. We extend the basic Kaczmarz method by introducing random row selection and varying step sizes to obtain faster convergence rates.

\textbf{Random Kaczmarz (RK) method:} The row selection probability is defined as $P(i=i_p)=\frac{\left\|a_{i_p}\right\|^2}{\left\|A\right\|_F^2}$, where $i_p\in[m]$. The row with the highest probability is selected first. When $A$ is fixed, rows with more nonzero elements have a higher probability of being selected. The algorithm proceeds as follows.

\begin{algorithm}
	\caption{Random Kaczmarz method}\label{RK algorithm}
	\begin{algorithmic}[1]
		\STATE  Choose an initial guess $x_0$
		\STATE	 For $p=0,1,2,\cdots,\ell-1$\\
				\quad a) Select $i_p\in\{1,2,\cdots,m\}$ with probability $Pr(row=i_p)=\frac{\left\|\alpha_{i_p}\right\|_2^2}{\left\|A\right\|_F^2}$\\
			\quad b) Update $x^{k+1}=x^k-w \frac{v_{i_p}-\alpha_{i_p}^Tx^{(p)}}{\left\|\alpha_{i_p}\right\|_2^2}\alpha_{i_p}$
		
		\STATE End for
	\end{algorithmic}
\end{algorithm}

\subsubsection{Constant Step Size for Random Kaczmarz Method}

In the Random Kaczmarz method, we can use a step size parameter to adjust the numerical value of each iteration. First, we can think of adding a step size $\alpha$ to each iteration based on the idea of gradient descent, which can control the value of each iteration. At the same time, when selecting rows, we do not adopt the method of selecting consecutive rows, but use the method of randomly selecting rows, in which the step size $\alpha=1$. It can be proved that this step size is the optimal step size.
\begin{theorem}{\rm\cite{refP32}}
According to the following theorem, it can be seen that this iterative method also has convergence, and the expected convergence rate is

$$E\left[\left\|x^{k+1}-x^*\right\|^2\right]\leqslant(1-\frac{\lambda_{min}^{nz}(A^TA)}{\left\|A\right\|_F})^k\left\|x^0-x^*\right\|^2,$$ where $\lambda_{min}^{nz}(\cdot)$ represents the minimum non-zero eigenvalue of the given matrix.\end{theorem} 

\begin{proof}
    
Consider $(x-x^*, (a_i^Tx-b_i)a_i)=(a_i^Tx-b_i)^2,$ in the interval $(0,2)$, for all $x^*$, where $Ax=b$. We get:

$$\begin{aligned}\left\|x^{k+1}-x^*\right\|^2&=\left\|x^k+\alpha\frac{(a_i^Tx^k-b_i)}{\left\|a_i\right\|}a_i-x^*\right\|^2\\&=\left\|x^k-x^*\right\|^2-
2\alpha\frac{(a_i^Tx^k-b_i)^2}{\left\|a_i\right\|^2}+\alpha^2\frac{(a_i^Tx^k-b_i)^2}{\left\|a_i\right\|^2}\\&
=\left\|x^k-x^*\right\|^2-\alpha(2-\alpha)\frac{(a_i^Tx^k-b_i)^2}{\left\|a_i\right\|^2}.\end{aligned}$$

Since the expectation formula is $E(X)=\sum_{i}p_ix_i$, where $X=(x_1,x_2,\cdots)$, the convergence of this algorithm is \begin{equation}\begin{aligned}\label{equ2}
	E\left[\left\|x^{k+1}-x^*\right\|^2\right]
	&=E\left(\left\|x^k-x^*\right\|^2-\alpha(2-\alpha)\frac{(a_i^Tx^k-b_i)^2}{\left\|a_i\right\|^2}\right)\\
	&=E\left(\left\|x^k-x^*\right\|^2\right)-\alpha(2-\alpha)E\left(\frac{(a_i^Tx^k-b_i)^2}{\left\|a_i\right\|^2}\right)\\ &=E\left(\left\|x^k-x^*\right\|^2\right)-\alpha(2-\alpha)\sum_i\frac{\left\|a_{i}\right\|^2}{\left\|A\right\|_F^2}\frac{(a_i^Tx^k-b_i)^2}{\left\|a_i\right\|^2}\\
	&=E\left(\left\|x^k-x^*\right\|^2\right)-\alpha(2-\alpha)\sum_i\frac{(a_i^Tx^k-b_i)^2}{\left\|A\right\|_F^2}\\
	&\leqslant\left\|x^k-x^*\right\|^2-\frac{\alpha(2-\alpha)}{\left\|A\right\|_F^2}\left\|Ax^k-b\right\|_F^2.\\\end{aligned}\end{equation}

By the Courant-Fischer theorem, we have for all matrices that

$$\left\|Ax\right\|^2\geqslant \lambda_{min}^{nz}(AA^T)\left\|x\right\|^2,$$
for all $x\in range(A^T)$. Then we can obtain

$$\left\|Ax^k-b\right\|^2=\left\|A(x^k-x_k^*)\right\|^2\geqslant\lambda_{min}^{nz}(AA^T)\left\|x^k-x^*\right\|^2.$$

Substituting the above equation yields:

$$E\left[\left\|x^{k+1}-x^*\right\|^2\right]
\leqslant(1-\frac{\alpha(2-\alpha)\lambda_{min}^{nz}(AA^T)}{\left\|A\right\|_F^2})\left\|x^k-x^*\right\|^2\leqslant(1-\frac{\lambda_{min}^{nz}(A^TA)}{\left\|A\right\|_F})^k\left\|x^0-x^*\right\|^2.$$

The term $\alpha(2-\alpha)$ in the above equation can be derived to obtain $\alpha=1(i.e.,\alpha=argmax_{\alpha}\alpha(2-\alpha))$ as the optimal choice (which achieves the smallest upper bound). This is also the basic Kaczmarz iteration method.\end{proof} 

Meanwhile, we can also know that $\alpha\in (0,2)$. In the next section, we prove that the correct choice of step size $\alpha_k>2$ can greatly accelerate the convergence rate.

	\subsubsection{Adaptive Step of Random Kaczmarz Method}
	 \cite{refP32}Extrapolation refers to the use of step size
$\alpha_k\geq2$ to move further along the straight line connecting the last iteration and the projected mean. For example, because the original iteration process could be slow, a variant of extrapolation was introduced with the following adaptive selection for step $\alpha_k$ $$\alpha_k\simeq 2\frac{\sum _{i\in J_k}\bar{w}_i(a_i^Tx^k-b_i)^2}{\left\|\sum_{i\in J_k}\bar{w}_i(a_i^Tx^k-b_i)a_i\right\|^2},$$ where we use the symbol $\bar{w}_i=w_i/\left\|a_i\right\|^2$ For convenience we define $0/0=1$, from Jenssen's inequality we can derive that $\alpha_k\geq 2$. 
$\alpha_k$ can be much greater than 2. In addition, sequences generated using an iterative process of an adaptive step k using extrapolation $x_k$ usually converge much faster than the same sequence $x_k$ generated in the underlying Kaczmarz method $\alpha_k\in(0,2)$.

We extrapolate the step size $(\alpha>2)$ and assign a weight to each row $w_k^i=w_i$ (the weight of the i-th row of the kth iteration), so the final iteration format is $$x^{k+1}=x^k-\alpha\left(\sum w_i\frac{a_i^Tx^k-b_i}{\left\|a_i\right\|^2}a_i\right), $$ where the weights of all rows are added to 1, and $0<
w_{min}\leqslant w_{max}<1$.

Define $\lambda_{max}^{block}=\max\lambda_{max}(A^Tdiag(\frac{1}{\left\|a_i\right\|^2})A),$
Starting from basic linear algebra, we obtain that the sum of the eigenvalues of a symmetric positive semidefinite matrix is equal to its trace. Therefore, if the rank of a matrix is at least 2, then its maximum eigenvalue is strictly less than that of the trace.  When we choose a random variable, make all the samples satisfy $|J|=\tau,\tau\in[1,m]$($\tau$  i.e. $|J|$ is the trace of the block matrix), then $\lambda_{max}^{block}<\tau,rank(A_J)\geq 2.$ Therefore, in this case, we use the overslack (extrapolation) step size, because usually $2w_{min}/w^2_{max}\lambda_{max}^{block}>2$. For example, for $w_i=1/\tau$, we get $2\tau/\lambda_ {max}^{block}>2$. Define the positive semidefinite matrix W as $$W=E_J\left[\sum_{i\in J}\frac{a_ia_i^T}{\left\|a_i\right\|^2}\right]=\sum_{[m]}p_i\frac{a_ia_i^T}{\left\|a_i\right\|^2}=\sum_{i\in[m]}p_i\frac{a_ia_i^T}{\left\|a_i\right\|^2}=A^Tdiag\left(\frac{p_i}{ \left\|a_i\right\|^2},i\in [m]\right)A.$$ The next theorem proves the convergence speed of the new algorithm with adaptive steps, which explicitly depends on the geometric properties of the system matrix A and its submatrices $A_J$. We denote the random average block Kaczmarz algorithm with adaptive step size as the $RABK$ algorithm.

	\begin{theorem}{\rm\cite{refP32}}
The sequence produced by the RaBK algorithm $\{x^k\}_{k\geqslant0}$ satisfies the weight $0<w_{min}\leqslant w_i\leqslant w_{max} <1$, for all $i\in[m]$, and the step size $\alpha=\frac{(2-\delta)w_{min}}{w^2_{max}\lambda_{max}^{block}}$ for all $\delta\in(0,1])$, then there is the following conclusion 
$$
E\left[\left\|x^k-x^*\right\|^2\right]\leqslant\left(1-\frac{(2-\delta)w^2_{min}\lambda_{min}^{nz}(W)}{w^2_{max}\lambda_{max}^{block}}\right)^k\left\|x^0-x^*\right\|^2.
$$
\end{theorem}

\begin{proof}Below we give proof of it.
	
	First we have $Ax^*=b$, then	$$\begin{aligned}
		\left\|x^{k+1}-x^*\right\|^2&=\left\|x^k-x^*-\alpha\left(\sum w_i\frac{a_i^Tx^k-b_i}{\left\|a_i\right\|^2}a_i\right)\right\|^2\\
		&=\left\|x^k-x^*-\alpha\left(\sum w_i\frac{a_i^Tx^k-a_i^Tx^*}{\left\|a_i\right\|^2}a_i\right)\right\|\\
		&=\left\|x^k-x^*-\alpha\left(\sum w_i\frac{a_i^T}{\left\|a_i\right\|^2}(x^k-x^*)a_i\right)\right\|^2\\
		&=\left\|(I_n-\alpha\left(\sum w_i\frac{a_i^Ta_i}{\left\|a_i\right\|^2}\right)(x^k-x^*)\right\|^2\\
		&=(x^k-x^*)^T\left(I_n-2\alpha\sum w_i\frac{a_ia_i^T}{\left\|a_i\right\|^2}+\alpha^2\left(\sum w_i\frac{a_ia_i^T}{\left\|a_i\right\|^2}\right)^2\right)(x^k-x^*),\\
	\end{aligned}$$where $a_ia_i^Tx=a_i^Txa_i$, $a_i^T$, $x$ is the column vector.
And
$$\begin{aligned}
	E\left[\sum w_i\frac{a_ia_i^T}{\left\|a_i\right\|^2}\right]&\succeq(minw_i)E\left[\sum \frac{a_ia_i^T}{\left\|a_i\right\|^2}\right]=w_{min}\sum p_i\frac{a_ia_i^T}{\left\|a_i\right\|^2}\\
	&=w_{min}A^Tdiag\left(\frac{p_i}{\left\|a_i\right\|^2}A\right)=w_{min}W,
\end{aligned}$$
then
$$E\left[\sum w_i\frac{a_ia_i^T}{\left\|a_i\right\|^2}\right] \succeq w_{min}W.$$
Also, for some $Q\succeq0,$ we have $Q^2\succeq\lambda_{max}(Q)Q,$ then we can get the expected upper bound $$\begin{aligned}
	E\left[\left(w_i\frac{a_ia_i^T}{\left\|a_i\right\|^2}\right)^2\right]
	&\preceq E\left[\lambda_{max}\left(\sum w_i\frac{a_ia_i^T}{\left\|a_i\right\|^2}\right)\left(\sum w_i\frac{a_ia_i^T}{\left\|a_i\right\|^2}\right)\right]\\
	&\preceq(max(w_i))E\left[\lambda_{max}\left(\sum \frac{a_ia_i^T}{\left\|a_i\right\|^2}\right)\left(\sum w_i\frac{a_ia_i^T}{\left\|a_i\right\|^2}\right)\right]\\
	&\preceq(max (w_i))E\left[\lambda_{max}\left(A^Tdiag(\frac{1}{\left\|a_i\right\|^2})A\right)
	\left(\sum w_i\frac{a_ia_i^T}{\left\|a_i\right\|^2}\right)\right]\\
	&\preceq(max(w_i)\lambda_{max}^{block}E\left[\sum\frac{a_ia_i^T}{\left\|a_i\right\|^2}\right]=w_{max}^2\lambda_{max}^{block}W.\\
\end{aligned}$$
We are $\left\|x^k-x^*\right\|^2=(x^k-x^*)^T\left(I_n-2\alpha\sum w_i\frac{a_ia_i^T}{\left\|a_i\right\|^2}+\alpha^2\left(\sum w_i\frac{a_ia_i^T}{\left\|a_i\right\|^2}\right)^2\right)(x^k-x^*)$ takes the expectation on both sides, then there are   $$\begin{aligned}E\left[\left\|x^k-x^*\right\|^2\right]&=E\left[(x^k-x^*)^T\left(I_n-2\alpha\sum w_i\frac{a_ia_i^T}{\left\|a_i\right\|^2}+\alpha^2\left(\sum w_i\frac{a_ia_i^T}{\left\|a_i\right\|^2}\right)^2\right)(x^k-x^*)\right]\\
	&=(x^k-x^*)^TE\left[\left(I_n-2\alpha\sum w_i\frac{a_ia_i^T}{\left\|a_i\right\|^2}+\alpha^2\left(\sum w_i\frac{a_ia_i^T}{\left\|a_i\right\|^2}\right)^2\right)\right](x^k-x^*)\\
	&\leqslant(x^k-x^*)^T\left(I_n-2\alpha w_{min}W+\alpha^2w^2_{max}\lambda_{max}^{block}W\right)(x^k-x^*).\\
\end{aligned}$$
We need to let $2\alpha w_{min}W-\alpha^2w^2_{max}\lambda_{max}^{block}W$ maximum,then $\alpha$ can take the value: $\alpha^*=\frac{w^2_{min}}{w^2_{max}\lambda_{max}^{block}}.$ we choose the step size  $\alpha=(2-\delta)w_{min}/w^2_{max}\lambda_{max}^{block}$, for all $\delta\in(0,1]$, then we can get $$E\left[x^{k+1}-x^*\right]\leqslant(x^k-x^*)^T\left(I_n-(2-\delta)\frac{w^2_{min}}{w^2_{max}\lambda_{max}^{block}W}\right)(x^k-x^*).$$
Also by the $Courant-Fisher$ theorem $$\left\|Ax\right\|^2\geqslant\lambda_{min}^{nz}(AA^T)\left\|x\right\|^2,$$ for all $x\in range(A^T)$. In addition we have $x-\pi_{x}(x)\in range(A^T)$.  let $D=diag(\frac{p_i}{\left\|a_i\right\|^2})$  non-singular, then $W=A^TDA$.

$$\begin{aligned}
	(x^k-x^*)^TW(x^k-x^*)&=\left\|D^{\frac{1}{2}}A(x^k-x^*)\right\|^2\geqslant\lambda_{min}^{nz}(A^TDA)\left\|x^k-x^*\right\|^2\\&=\lambda_{min}^{nz}(W)\left\|x^k-x^*\right\|^2\\
\end{aligned}$$

For the above formula,  we have:$$\begin{aligned}
	E\left[\left\|x^{k+1}-x^*\right\|\right]&\leqslant(x^k-x^*)^T\left(I_n-(2-\delta)\frac{w^2_{min}}{w^2_{max}\lambda_{max}^{block}W}\right)(x^k-x^*)\\
	&\leqslant\left(1-(2-\delta)\frac{w_{min}^2}{w^2_{max}\lambda_{max}^{block}}\lambda_{min}^{nz}(W)\right)\left\|x^k-x^*\right\|^2.\\
 &\leqslant\left(1-\frac{(2-\delta)w^2_{min}\lambda_{min}^{nz}(W)}{w^2_{max}\lambda_{max}^{block}}\right)^k
\left\|x^0-x^*\right\|^2.
\end{aligned}$$

\end{proof}

\begin{algorithm}[H]
	\caption{Adaptive Step Kaczmarz algorithm flow}
	\begin{algorithmic}[1]
		\STATE	Take the initial value of $x_0$
		\STATE	calculate $\alpha_k=(2-\delta)L_k$,
		$$L_k=\begin{cases}
			\frac{\sum_{i\in J_k}\bar{w}_i^k(a_i^Tx^k-b_i)^2}{\left\|\sum_{i\in J_k}\bar{w}_i^k(a_i^Tx^k-b_i)a_i\right\|^2} \quad if\;\;i\in J_k\;s.t.a_i^Tx_k-b_i\neq0,\\
			\frac{1}{\lambda_{max}(A_{J_k}^Tdiag(\bar{w}_i^k,i\in J_k)A_{J_k})}\quad otherwise.
		\end{cases}$$
		and $w_i=1/\tau$, $\bar{w}_i=\frac{w_i}{\left\|a_i\right\|_2},$ where $\tau$  is the sum of the diagonal elements representing the matrix
		\STATE	for $k=1,2\cdots$ do
	$$x^{k+1}=x^k-\alpha_k (\sum w_i\frac{a^Tx^k-b_i}{\left\|a_i\right\|^2}a_i)$$
		\STATE	end for
	\end{algorithmic}
\end{algorithm}

\section{Complexity Analysis}
We combine the three internal iteration methods (RPCG, ADI, Kaczmarz) above to combine the methods of the BA-GMRES algorithm as the NM-BA-GMERES method. For the above algorithm, we can get the algorithm flowchart as:

\begin{algorithm}[H]
	\caption{NM-BA-GMRES algorithm flow}\label{tab:算法流程}
	\begin{algorithmic}[1]
		\STATE	Take the initial value $x_0$, $r=b-Ax_0$
		\STATE	$v_1=r/\left\|r\right\|_2$	
		\STATE	for $k=1,2\cdots$ until convergence do
		\STATE   Apply inner iterative methods to solve $Az=v_k$
		\STATE	Compute $B^{\ell_k}=(M^{-1}N)^kz_0+\sum_{i=0}^{k-1}(M^{-1}N)^iM^{-1}$
		where $A=M-N$,$\ell_{max}$ is the maximum number of iterations obtained by solving $Az=v_k$ for three method, $\ell_k=\min\{l_{max},l\}.$,$\ell$ makes $\left\|BAv_k-v_k\right\|_2\leqslant\left\|v_k\right\|_2$ minimal
		\STATE	Solve $z_k=B^{\ell_k}v_k$
		\STATE	 $w_k=BAz_k$
		\STATE	for i=1,2$\cdots,k$,do
		\STATE	$h_{i,k}=w_k^Tv_i,w_k=w_k-h_{i,k}v_i$
		\STATE	end for
		\STATE	$h_{k+1,k}=\left\|w_k\right\|_2,v_{k+1}=w_k/h_{k+1,k}$
		\STATE	end for
		\STATE	$y_k=argmin_{y\in R^k}\left\|\beta e_1-\bar{H}_ky\right\|_2,$
		\quad$u_k=[v_1,v_2,\cdots,v_k]y_k,$
		\quad where$\bar{H}_k=\{h_{ij}\}_{1\leqslant i\leqslant k+1;1\leqslant j\leqslant k}$
		\STATE	$x_k=x_0+u_k$
	\end{algorithmic}
\end{algorithm}

Let the coefficient matrix be 
$A\in \mathbb{R}^{n\times n}$. From the algorithm described above, it is evident that the complexity of the algorithm mainly lies in solving 
$Az=v_k$,
  using inner iteration methods, preprocessing matrix 
$B^{l_k}$
 , and the Arnoldi process in GMRES for solving $$y_k=argmin_{y\in R^k}\left\|\beta e_1-\bar{H}_ky\right\|_2.$$
 
The complexity of Step 1 is reflected in matrix-vector multiplication, hence the complexity is 
$\mathcal{O}(mn)$;
The complexity of Step 2 lies in vector normalization, involving division and norm calculations, with a complexity of 
$\mathcal{O}(n)$;
If Step 4 adopts the Restricted Preconditioned Conjugate Gradient (RPCG) method and assuming convergence within 
$l_k$ steps, calculating the time complexity from Step 3 to Step 4: Since the RPCG method iterates 
$l_{max}$ times, the time complexity is 
$\mathcal{O}(l_{max}mn)$. Therefore, the total time complexity is 
$\mathcal{O}(l_{max}mnl_k)$.

If Step 4 uses Alternating Direction Iteration and assuming convergence within 
$l_k$
  steps, calculating the time complexity from Step 3 to Step 4: When using Alternating Direction Iteration, solving two linear equation systems is required per iteration, totaling 
$2l_k$ linear equation solves. The time complexity of solving an 
$m\times n$ linear equation system each time is 
$\mathcal{O}(mn^2)$, thus the total time complexity of Alternating Direction Iteration is 
$\mathcal{O}(l_kmn^2)$.

If Step 4 employs the Kaczmarz method, the time complexity is 
$\mathcal{O}(l_kl_{max}n)$, where 
$l_k$ represents the number of iterations in the GMRES algorithm. In the Kaczmarz method, each iteration processes all equations, requiring 
$n$ multiplications in total, and since the algorithm converges in 
$l_{max}$
  iterations, the time complexity per iteration is 
$\mathcal{O}(l_{max}n)$.

Step 5 computes 
$B_{l_k}$,
  with a time complexity of 
$\mathcal{O}(l_kn^2)$;
Step 6 solves 
$z_k=B^{l_k}v_k$,
  with a time complexity of 
$\mathcal{O}(l_kn)$;
Step 7 calculates 
$w_k=BAz_k$
  with a time complexity of 
$\mathcal{O}(n^2)$;
Steps 8 to 12 involve loop operations, with a total time complexity of 
$\mathcal{O}(l_k^2n)$;
Step 13: Solving 
$y_k=argmin_{y\in R^k}\left\|\beta e_1-\bar{H}_ky\right\|_2, $ has a time complexity of 
$\mathcal{O}(l_k^2)$;
Step 14: Computing 
$x_k=x_0+u_k$,
  takes 
$\mathcal{O}(l_kn)$; time.
Considering the time complexities of each step, the overall time complexity of the NM-BA-GMRES algorithm is given by $$\mathcal{O}(l_k^2n+f(l_k,n)+l_kn+n^2),$$ where 
$f(l_k,l_{max},m,n),l_kn,n^2)$ represent the complexities of the three inner iteration algorithms.

It is important to note that this time complexity only considers the main computational operations and not other factors (such as matrix storage and data transfer). In practical applications, these factors need to be considered as well.
From the analysis of the complexity above, it is evident that the algorithm can reduce complexity to some extent, which is one of its advantages.

\section{Numerical Example}

Below is an example of the numerical values of the algorithm, unless otherwise specified, the test matrices are all from the matrix market. We mark the following symbols.
\begin{itemize}
\item no-pre: iteration without preprocessing;

\item pre-1: Kacamarz inner iteration plus BA-GMRES outer iteration with a step size of 1;

\item pre-adapt: Kacamarz in-iteration plus BA-GMRES out-of-the-BA iteration and adaptive step;

 \item pre-adapt-r: Kacamarz in-circuit plus BA-GMRES out-of-the-BOX iteration and adaptive step size and random row selection;

\item ADI-pre: ADI approach as internal iteration, BA-GMRES as external iteration; 

\item PCG-pre: PCG method as internal iteration, BA-GMRES as external iteration; 

\item rpcg-pre: rpcg method as internal iteration, BA-GMRES as external iteration;

\item -: If we set the number of iterations to 300 times, that is, the error does not converge to $10^{-6}$ within 300 times, then the iteration fails.
\end{itemize}
\subsection{Numerical Example 1}
The test matrix for the numerical example is $dwt\_198$ and $can\_\_229$. The GMRES method without pretreatment is compared with the Kaczmarz-BA-GMRES and Kaczmarz method with step 1, Kaczmarz-BA-GMRES and Kaczmarz method with adaptive step size, The row selection method of Kaczmarz-BA-GMRES and Kaczmarz method is random row selection. The results obtained are as follows.

\begin{center}
	\begin{tabular}{c|c|c|c|c}
		\hline
		$dwt\_198$	& no-pre &pre-1&pre-adapt&pre-adapt-r  \\
		\hline
		error	&8.236984e-07&-&9.977501e-07&9.657820e-07
		\\
		\hline
	 iteration	&86  &300&79 &86\\
		\hline
		time	& 0.1123745& 30.105613&0.163483&1.966090 \\
		\hline
	\end{tabular}
\end{center}
In the above numerical example, we can see that pre-adapt has fewer iterations than no-pre, and the other two preprocessing methods perform slightly worse than no-pre methods. Therefore, specific problems need to be analyzed in the pretreatment.

The following figure is a line chart of error comparison for the four methods of $dwt\_198$.

\begin{figure}[H]
	\centering
	\includegraphics[width=0.7\linewidth]{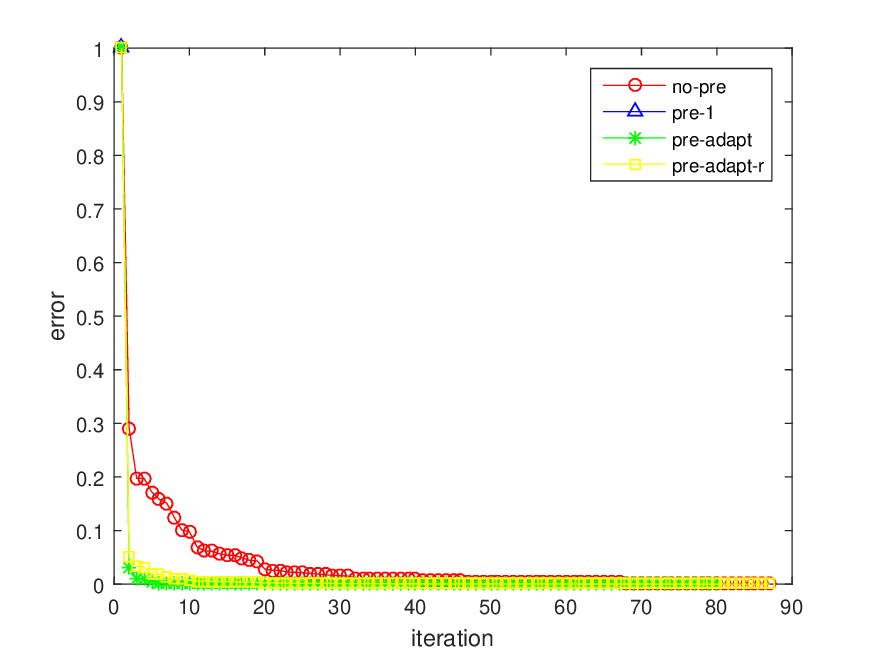}
	\caption{Error comparison plot of the sparse matrix 'dwt\_198'}
\end{figure}

\begin{center}
	\begin{tabular}{c|c|c|c|c}
		\hline
		$can\_\_229$	& no-pre &pre-1 &pre-adapt&pre-adapt-r  \\
		\hline
		error	& 4.190692e-07 & 6.507735e-11&6.282670e-08&2.744568e-10\\
		\hline
		iteration	&35&1& 36&2
		\\
		\hline
		time	&0.007801& 12.845967&0.182913&467.691122 \\
		\hline
	\end{tabular}
\end{center}

In the above numerical example, we can see that compared to no-pre, pre-1 and pre-adapt-r perform better in terms of number of iterations, and the other two preprocessing methods are not so good.

The following figure is a line chart of error comparison for the four methods of $can\_\_229$.

\begin{figure}[H]
	\centering
	\includegraphics[width=0.7\linewidth]{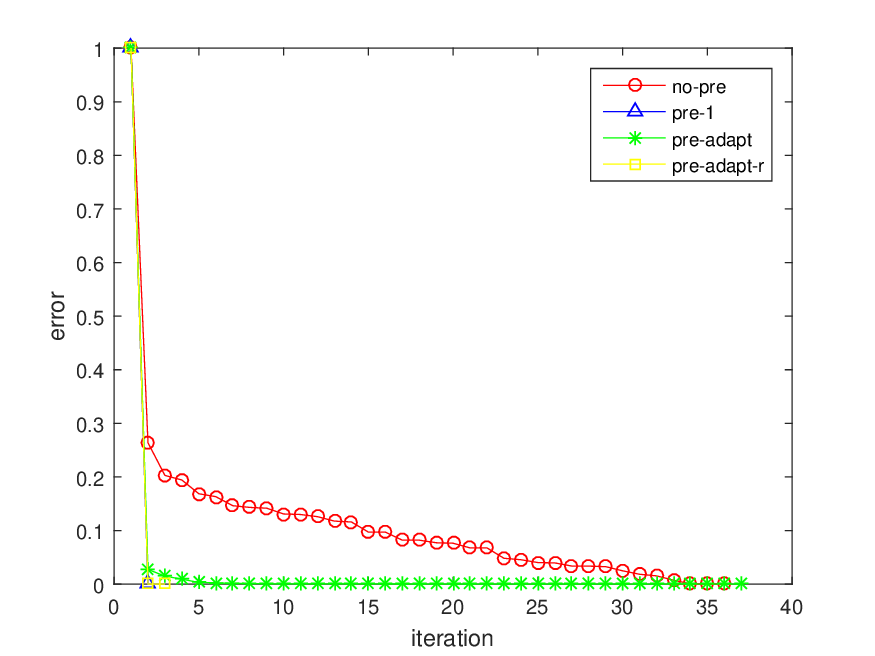}
	\caption{Error comparison plot of the sparse matrix can\_\_229}
	\label{fig:can229}
\end{figure}
\subsection{Numerical Example 2}

We compared four methods: GMRES without preprocessing, Kaczmarz-BA-GMRES with Kaczmarz method using step size 1, Kaczmarz-BA-GMRES with Kaczmarz method using adaptive step size, and the basic ADI iteration as the inner iteration and BA-GMRES as the outer iteration. The test matrix used was 494\_bus, and the results obtained are as follows.

\begin{center}
	\begin{tabular}{c|c|c|c|c}
		\hline
		&no-pre&pre-1 &pre-adapt&ADI-pre  \\
		\hline
		error	 &9.189161e-07&9.870971e-07&9.330144e-07&8.914631e-07 \\
		\hline
		iteration&	281&279& 288&252
		\\ 
		\hline
		time&0.238567	&0.254157 &14.141891&0.233930  \\
		\hline
	\end{tabular}
\end{center}
In the above numerical example, ADI-pre outperforms no-pre in terms of iteration time and number of iterations.

\begin{figure}[H]
	\centering
	\includegraphics[width=0.7\linewidth]{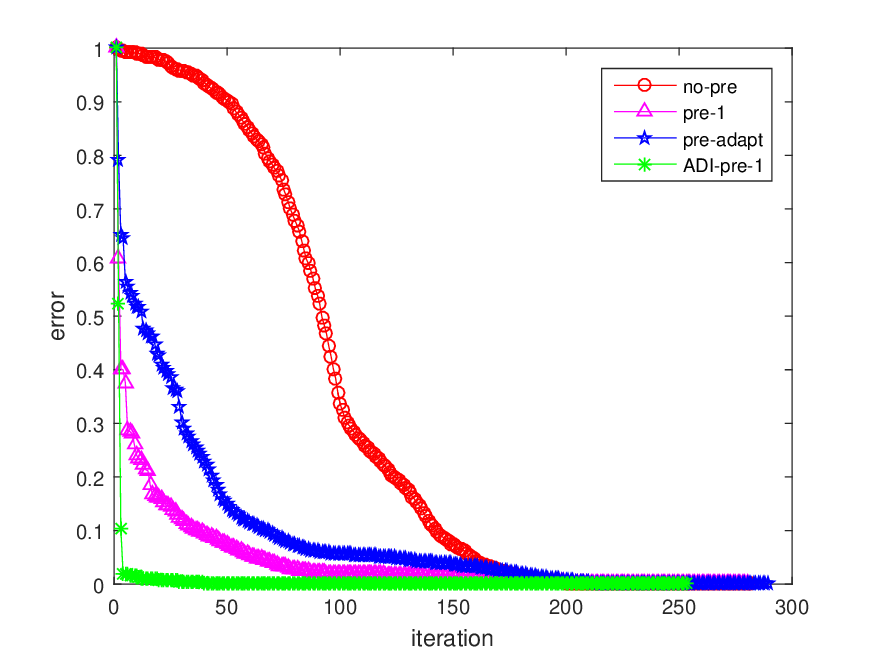}
	\caption{Error comparison plot of the sparse matrix 494\_bus}
	\label{fig:adipre-1pre-ada}
\end{figure}

The test matrix of the numerical example is bcspwr02. The GMRES method without preprocessing is compared with the basic ADI iteration format as the inner iteration and BA-GMRES as the outer iteration. The results obtained are as follows.

\begin{center}
	\begin{tabular}{c|c|c|c|c}
		\hline
		&no-pre&pre-1 &pre-adapt&ADI-pre  \\
		\hline
		error	 &6.660309e-17&2.159542e-07&3.652257e-08&2.864237e-16\\
		\hline
		iteration&	86&1& 2&42
		\\ 
		\hline
		time&0.019226	&0.269463&5.383092&0.021467  \\
		\hline
	\end{tabular}
\end{center}

From the above numerical example, we can see that ADI-pre is significantly better than no-pre in terms of number of iterations and iteration time.

\begin{figure}[H]
	\centering	\includegraphics[ width=0.7\linewidth]{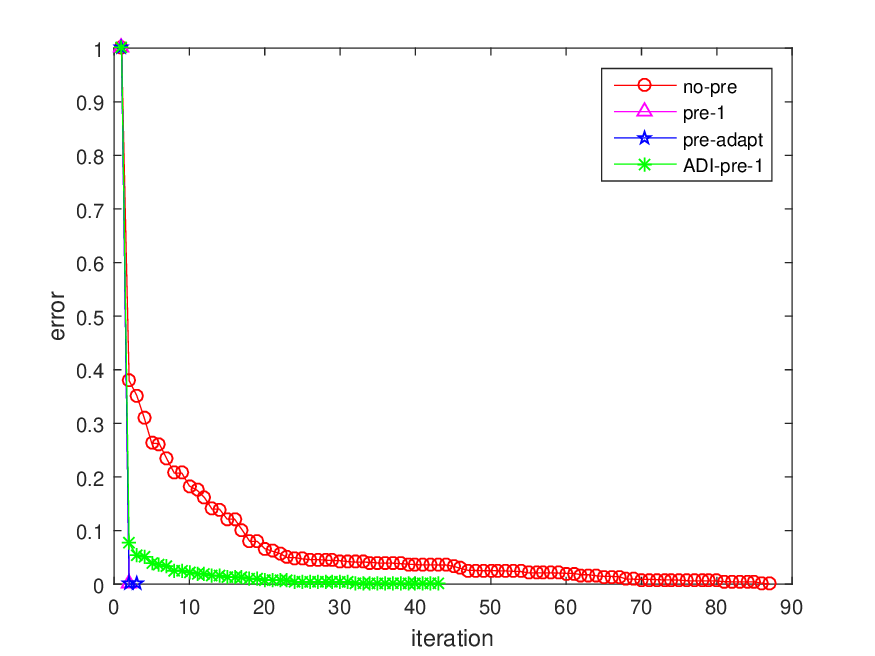}
	\caption{Error comparison plot of the sparse matrix bcspwr02}
	\label{fig:}
\end{figure}

\subsection{Numeric Example 3}\label{ex3}

We compare the preconditioned conjugate gradient method (PCG) as the inner iteration and BA-GMRES as the outer iteration with the GMRES method without pretreatment, and obtain the following numerical results. The test matrices used are $nos1$ and $494\_bus$.

\begin{center}
	\begin{tabular}{c|c|c}
		\hline
		$nos1$	&no-pre &PCG-pre  \\
		\hline
		error	 &6.803559e-09
		&9.844183e-07 \\
		\hline
		iteration	&198& 101
		\\
		\hline
		time	&0.150153 &0.113811\\
		\hline
	\end{tabular}
\end{center}

\begin{center}
	\begin{tabular}{c|c|c}
		\hline
		$494\_bus$	&no-pre &PCG-pre  \\
		\hline
		error	 &9.189161e-07
		&8.914631e-07 \\
		\hline
		iteration	&281& 252
		\\
		\hline
		time	&0.489451 &0.269410 \\
		\hline
	\end{tabular}
\end{center}

From the above two numerical examples, we can see that PCG-pre is better than no-pre in terms of iteration time and number of iterations.

We use the RPCG method as the internal iteration and BA-GMRES as the outer iteration to obtain the following results. The matrices we choose are the random matrix and the tridiagonal matrix $$\begin{bmatrix}
	10&2 &0&\cdots&0\\
	2&10&2&\cdots&0\\
	0&2&10&\cdots&0\\
	\vdots&\vdots&\vdots&&\vdots\\
	0&0&0&\cdots&10\\
\end{bmatrix}.$$
\begin{center}
	\begin{tabular}{c|c|c|c|c}
		\hline
		Tridiagonal matrix	& no-pre &pre-1&pre-adapt&rpcg-pre  \\
		\hline
		error	&2.864893e-07&2.864928e-07&2.864893e-07&5.503294e-07
		\\
		\hline
		iteration	&8  &8&8 &7\\
		\hline
		time	&0.002582& 0.078088&65.722983&0.019600  \\
		\hline
	\end{tabular}
\end{center}

In the above numerical example, it can be seen that the rpcg-pre method is better than no-pre in terms of number of iterations and iteration time.

\begin{figure}[H]
	\centering
	\includegraphics[width=0.7\linewidth]{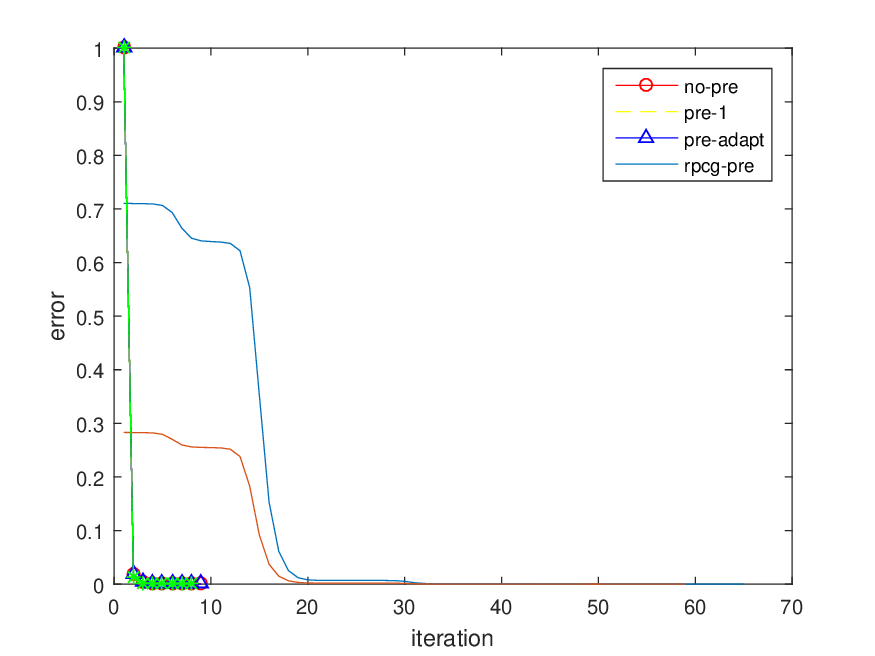}
	\caption{ Error comparison plot of the sparse matrix three diagonal matrices}
	\label{fig:rpcg-pre}
\end{figure}

\section{Conclusion}

In this paper, a method for solving a system of linear equations is studied, using three different internal iterative methods to calculate the vectors required for the Krylov subspace, and then using BA-GMRES to calculate the solution of the equation system. We give three numerical examples to test the performance of each of the three methods. It is concluded that the NM-BA-GMERES method after different pretreatment is superior to the GMRES method in terms of convergence rate.

\section{Statements and Declarations}
No potential conflict of interest was reported by the authors.

\begin{thebibliography}{99}
 	\bibitem{refP1} K. Morikuni, K. Hayami. Convergence of inner-iteration GMRES methods for rank-deficient least squares problems[J]. SIAM J. Matrix Anal. Appl., 2015, 36(1): 225-250.
\bibitem{refP2}Z.-Z Bai. Sharp error bounds of some Krylov subspace methods for non-Hermitian linear systems[J]. Appl. Math. Comput., 2000, 109(2-3): 273-285.
\bibitem{refP3}M. R. Hestenes, E. Stiefel. Methods of conjugate gradients for solving linear systems[J]. J. Research Nat. Bur. Standards, 1952, 49(6): 409-436.
\bibitem{refP4}C. C. Paige,  M. Saunders. LSQR: An algorithm for sparse linear equations and sparse least squares[J]. ACM Trans. Math. Software, 1982, 8(1): 43-71.
\bibitem{refP6} K. Hayami, M. Sugihara. A geometric view of Krylov subspace methods on singular systems[J]. Numer. Linear Algebra Appl., 2011, 18(3): 449-469.
\bibitem{refP7}P. N. Brown, H. F. Walker. GMRES on (nearly) singular systems[J]. SIAM J. Matrix Anal. Appl., 1997, 18(1): 37-51.
\bibitem{refP10} K. Hayami, J.-F. Yin, T. Ito. GMRES methods for least squares problems[J]. SIAM J. Matrix Anal. Appl., 2010, 31(5): 2400-2430.
\bibitem{refP11} Y. Saad. Iterative Methods for Sparse Linear Systems[M]. Philadelphia: SIAM, 2003.
\bibitem{refP12}R. Ansorge.  Connections between the Cimmino-method and the Kaczmarz-method for the solution of singular and regular systems of equations[J]. Computing, 1984, 33(3): 367-375.
\bibitem{refP13} Z.-Z Bai, Liu, X.-G. On the Meany inequality with applications to convergence analysis of several row-action iteration methods[J]. Numer. Math., 2013, 124(2): 215-236.
\bibitem{refP14} P.C. Hansen,  K. Hayami,  K. Morikuni. GMRES Methods for Tomographic Reconstruction with an Unmatched Back Projector[J]. J. Comput. Appl. Math., 2021, 413: 114352.
\bibitem{refP15}W. Xu,  N. Zheng,  K. Hayami. Jacobian-Free Implicit Inner-Iteration Preconditioner for Nonlinear Least Squares Problems[J]. J. Sci. Comput., 2016, 68(3): 1055-1081.
\bibitem{refP16}J. Yin, et al. Preconditioned Krylov subspace method for the solution of least-squares problems[J]. PAMM, 2007, 7(1): 2020151-2020152.
\bibitem{refP17} L. Zhao,  T. Huang,  L. Deng. Krylov subspace methods with deflation and balancing preconditioners for least squares problems[J]. J. Appl. Anal. Comput., 2019, 9(1): 57-74.
\bibitem{refP18} T. Strohmer,  R. Vershynin. A Randomized Solver for Linear Systems with Exponential Convergence[M]. Berlin, Heidelberg: Springer Berlin Heidelberg, 2006.
\bibitem{refP19} T. Strohmer, R. Vershynin. A randomized Kaczmarz algorithm with exponential convergence[J]. J. Fourier Anal. Appl., 2009, 15(2): 262-278.
\bibitem{refP20}Z.-Z Bai. Restrictively preconditioned conjugate gradient methods for systems of linear equations[J]. IMA J. Numer. Anal., 2003, 23(4): 561-580.
\bibitem{refP21}Z.-Z Bai,  W.-T. Wu. On greedy randomized Kaczmarz method for solving large sparse linear systems[J]. SIAM J. Sci. Comput., 2018, 40(1): A592-A606.
\bibitem{refP22} C. Popa. Convergence rates for Kaczmarz-type algorithms[J]. Numer. Algorithms, 2018, 79(1): 1-17.
\bibitem{refP24} T. Strohmer,  R. Vershynin. A randomized Kaczmarz algorithm with exponential convergence[J]. J. Fourier Anal. Appl.,  2009, 15(2): 262-278.
\bibitem{refP25}Y. S. Du, K. Hayami, N. Zheng, et al. Kaczmarz-type inner-iteration preconditioned flexible GMRES methods for consistent linear systems[J]. SIAM J. Sci. Comput., 2021, 43(5): S345-S366.
\bibitem{refP26} D. W. Peaceman, Rachford, H. H. Jr.  The numerical solution of parabolic and elliptic differential equations[J]. J. Soc. Ind. Appl. Math., 1955, 3(1): 28-41.
\bibitem{refP27} Z.-Z. Bai, G. H. Golub, M. K. Ng. Hermitian and skew-hermitian splitting methods for non-hermitian positive definite linear systems[J]. SIAM J. Matrix Anal. Appl., 2003, 24(3): 603-626.
\bibitem{refP28} M. Benzi, G. H. Golub, A Preconditioner for Generalized Saddle Point Problems[J]. SIAM J. Matrix Anal. Appl., 2004, 26(1): 20-41.
\bibitem{refP31} C. Greif,  J. Varah. Iterative solution of cyclically reduced systems arising from discretization of the three-dimensional convection-diffusion equation[J]. SIAM J. Sci. Comput., 1998, 19(6): 1918-1940.
\bibitem{refP32}I. Necoara.  Faster randomized block Kaczmarz algorithms[J]. SIAM J. Matrix Anal. Appl.,  2019, 40(4): 1425-1452.
\bibitem{refP33}Xu Shufang, Gao Li, ZHANG Pingwen. Numerical Linear Algebra [M]. Beijing: Peking University Press, 2000.
\bibitem{refP34} S. Kaczmarz. Approximate solution of systems of linear equations[J]. Int. J. Control, 1993, 57(6): 1269-1271.
	\end{thebibliography}

	\end{document}